\documentclass[11pt, reqno]{amsart}
\usepackage{amsmath,amssymb,mathtools}
\usepackage{microtype}
\usepackage[margin=1.08in]{geometry}
\usepackage{enumitem}
\usepackage{booktabs,longtable,array}
\usepackage{xcolor}
\usepackage[colorlinks=true,linkcolor=blue!55!black,citecolor=blue!55!black,urlcolor=blue!55!black]{hyperref}
\usepackage{fancyhdr}

\numberwithin{equation}{section}
\newtheorem*{Brayconjecture}{Bray's conjecture}
\newtheorem{theorem}{Theorem}[section]
\newtheorem{proposition}[theorem]{Proposition}
\newtheorem{lemma}[theorem]{Lemma}
\newtheorem{corollary}[theorem]{Corollary}
\theoremstyle{remark}
\newtheorem{remark}[theorem]{Remark}

\DeclareMathOperator{\Ric}{Ric}

\newcommand{\Sph}{\mathbb S}
\newcommand{\dbarV}[1]{d\overline V_{#1}}

\newcommand{\eps}{\varepsilon}
\newcommand{\MW}{\mathrm{MW}}
\newcommand{\ent}{\mathrm{ent}}

\newcommand{\W}{\mathcal{W}}


\author{Xumin Jiang}
\address{Xumin Jiang, School of Sciences, Great Bay University, Dongguan 523000, China}
\email{xjiang@gbu.edu.cn}

\author{Mingxiang Li}
\address{Mingxiang Li, Department  of Mathematics \& Institute of Mathematical Sciences, The Chinese University of Hong Kong}
\email{limx@smail.nju.edu.cn}

\author{Zhehui Wang}
\address{Zhehui Wang, School of Sciences, Great Bay University, Dongguan 523000, China}
\email{wangzhehui@gbu.edu.cn}

\title{On the proof of Bray's conjecture}

\date{}

\begin{document}

\begin{abstract}

Let $(M^n,g)$ be a connected, complete, smooth  Riemannian manifold with  dimension $n\geq 3$.  There exists a positive constant $\varepsilon_n<1$ such that,  if   Ricci curvature $\Ric_g\geq \varepsilon_n(n-1)g$ and the scalar curvature $R_g\geq n(n-1)$, then the volume $V_g(M^n)$ is less than or equal to the volume of standard $n$-sphere. This confirms a conjecture by Bray in 1997.
\end{abstract}

\maketitle

\section{Introduction}
Let $(M^n,g)$ be a connected, complete, smooth Riemannian  manifold. The classical Bishop's volume comparison theorem shows that if Ricci curvature satisfies 
\begin{equation}\label{Ricci-geq-n-1}
    \Ric_g\geq (n-1)g,
\end{equation}
 then the volume $V_g(M^n)$ satisfies a sharp upper bound
 \begin{equation}\label{volume-upper-bound}
   V_g(M^n)\leq |\mathbb{S}^n|  
 \end{equation}
where $|\Sph^n|$ denotes the volume of standard $n$-sphere. Such a volume comparison theorem plays a fundamental role in differential geometry. It is natural to ask whether the condition \eqref{Ricci-geq-n-1} can be relaxed. It is known that if we only assume the scalar curvature $R_g$ satisfies
\begin{equation}\label{R_g-lower-bound}
    R_g \geq n(n-1),
\end{equation}
then the volume estimate \eqref{volume-upper-bound} may no longer hold (see \cite{CEM}). In his thesis \cite{Bray}, Bray asked whether \eqref{volume-upper-bound} remains true if we slightly relax the lower bound on the Ricci tensor while still assuming \eqref{R_g-lower-bound}. We now state Bray's conjecture precisely (see Conjecture 4 in \cite{Bray}).
\begin{Brayconjecture}\label{Bray's-conjecture-origin}
    Let $(M^n,g)$ be a connected, complete, smooth Riemannian manifold with $n\geq 3$. There exists a positive constant $\varepsilon_n<1$ such that, if the Ricci tensor satisfies 
    \begin{equation}\label{Bray'scondition-Ricci}
      \Ric_g\geq \varepsilon_n(n-1)g  
    \end{equation}
    and the scalar curvature satisfies 
    \begin{equation}\label{Bray'scondion-for-R_g}
       R_g\geq n(n-1), 
    \end{equation}
 then
 the volume satisfies 
 \begin{equation}\label{Bray's-volume-comparison}
     V_g(M^n)\leq |\Sph^n|.
 \end{equation}
\end{Brayconjecture}

There are many excellent works contributing to this conjecture. First, Bray \cite{Bray} confirmed this conjecture in dimension three, and numerical tests suggest $0.134 < \varepsilon_3 < 0.135$; this is known as Bray's football theorem. Later, Gursky and Viaclovsky \cite{Gursky-Via} showed that $\varepsilon_3 \leq 1/2$, and Brendle \cite{Brendle} proved the rigidity for the case $\varepsilon_3 = 1/2$. Yuan \cite{Yuan} proved it under the assumption that the metric is sufficiently close to the standard sphere in the $C^2$ sense. Recently, Zhang \cite{Zhang} proved it under an additional assumption that $\operatorname{Ric}_g \leq C g$ for some positive constant $C$, or that $(M^n,g)$ is axisymmetric. Very recently, Kwong \cite{Kwong} showed that if, for some constant $\varepsilon \geq 0$,
\begin{equation}\label{Kwong's-condition}
    \operatorname{Ric}_g \geq (n-1)g \quad \text{and} \quad R_g \geq n(n-1)(1+\varepsilon),
\end{equation}
then the volume satisfies 
\begin{equation}\label{Kwong'sresult}
   V_g(M^n) \leq (1 + n \varepsilon)^{-\frac{1}{2}} |\mathbb{S}^n|.
\end{equation}
Up to scaling, Bray's conjecture states that the volume should have the upper bound
\begin{equation}\label{Bray's-conjecture}
     V_g(M^n) \leq (1 + \varepsilon)^{-\frac{n}{2}} |\mathbb{S}^n|.
\end{equation}
A simple Taylor expansion shows that Kwong's bound \eqref{Kwong'sresult} agrees with Bray's conjecture \eqref{Bray's-conjecture} up to terms of order $\varepsilon^2$.

Our main result in this paper is stated as follows.
\begin{theorem}\label{thm:main-theorem}
   The  Bray's conjecture holds true. {Moreover, the equality in \eqref{Bray's-volume-comparison} holds if and only if $(M^n, g)$  is isometric to standard sphere.}
\end{theorem}

Let us briefly introduce the strategy of the argument. Under the normalization \eqref{Kwong's-condition}, we consider Perelman's entropy $\W_g(u,\tau)$ and its infimum $\nu(g)$. Using the spherical P\'olya–Szeg\"o theorem and Beckner's inequality, we show that if \eqref{Bray's-conjecture} fails, then $\nu(g)$ is strictly greater than the entropy limit of the $n$-sphere. However, by applying Ma–Wang's result \cite{MaWang}, we prove that this cannot happen. 

This paper is organized as follows. In Section \ref{sec:pre}, we review several known results from the literature. In Section \ref{sec:keylemma}, we present some key lemmas. Section \ref{sec:proof} is devoted to completing the proof of Theorem \ref{thm:main-theorem}. Finally, in Section \ref{sec:application}, we employ a Bray-type volume comparison argument to derive the corresponding volume comparison theorem in terms of the $Q$-curvature.

\vspace{3em}
{\bf Disclosure on AI assistance.} The authors used AI-assisted tools, principally ChatGPT. The authors wrote and verified all theorem statements, proofs, and they take full responsibility for the contents of the paper.

{\bf Acknowledgement.} This is our submitted version. We notice that Fan and Yu \cite{fanyu} independently proved the rigidity in the equality case concurrently.  X. Jiang is partially supported by Guangdong Basic and Applied Basic Research Foundation (Grant No. 2024A1515140083), Guangdong Provincial Project (Grant No. 2024QN11X020),
and NSFC Grant 12571211. Z. Wang is partially supported by NSFC Grant 12401247 and Guangdong Basic and Applied Basic Research Foundation (Grant No.\,{2025A1515140243}).

\section{Preliminaries}\label{sec:pre}

In this section, we introduce four key ingredients that will be used in the proof of Bray’s conjecture.  For the readers' convenience, we repeat their statements below.

Let $(M^n,g)$ be a connected, closed, smooth Riemannian manifold. 
For brevity, we suppress the superscript  $n$ and simply denote the manifold by $M$ and set the normalized volume form 
$$\; \dbarV g:=V_g(M)^{-1}\; d V_g.$$
{The average scalar curvature is denoted by
\begin{equation}\label{eq:average-scalar}
\overline R_g:=\frac1{V_g(M)}\int_MR_g\,dV_g.
\end{equation}}
\subsection{\bf Perelman’s $\W$-entropy}
For a constant $\tau>0$ and a nonnegative function $v\in H^1(M)$ satisfying
\begin{equation}\label{eq:L2-normalization}
\int_Mv^2\,\dbarV g=1,
\end{equation}
we consider Perelman’s entropy \cite{Perelman} under the normalized volume form,
\begin{align}
\W_g(v,\tau):={}&4\tau\int_M|\nabla v|_g^2\,\dbarV g
+\tau\int_MR_gv^2\,\dbarV g\notag\\
&-\int_Mv^2\log v^2\,\dbarV g
+\log V_g(M)-\frac n2\log(4\pi\tau)-n.\label{eq:W-probability}
\end{align}
It is equivalent to Perelman's $f$-variable form entropy in \cite{Perelman} by writing
\begin{align}
v=V_g(M)^{1/2}(4\pi\tau)^{-n/4}e^{-f/2}.
\end{align}
 Meanwhile, as in \cite{Perelman}, we define
\begin{equation}\label{eq:mu-nu-definitions}
\mu(g,\tau):=\inf\left\{\W_g(v,\tau):
v\in C^\infty(M),\ v\ge0,\ \int_Mv^2\,\dbarV g=1\right\},
\end{equation}
and 
\begin{equation}\label{eq:mu-nu-definitions-2}
\nu(g):=\inf_{\tau>0}\mu(g,\tau),
\end{equation}
with the convention that $\W_{g_{\Sph^n}}(v,\tau)$, $\mu(g_{\Sph^n},\tau)$ and $\nu(g_{\Sph^n})$ denote the corresponding quantities defined on the standard sphere $(\Sph^n,g_{\Sph^n})$.

With the help of some standard argument, the same $\mu(g,\tau)$ is also obtained from
nonnegative normalized $H^1(M)$ functions,
\begin{equation}\label{eq:H1-mu}
\mu(g,\tau)=\inf\left\{W_g(v,\tau):v\in H^1(M),\ v\geq0,
\ \int_Mv^2\,\dbarV g=1\right\}.
\end{equation}

We will use the following facts about the Perelman's entropy.
\begin{proposition}\label{prop:invariance}
Let $\mu$ be defined as in \eqref{eq:mu-nu-definitions}.
\begin{itemize}
    \item[(1)] For every constant $c>0$, diffeomorphism $\Phi$, and constant $\tau>0$, one has
\begin{equation}\label{eq:scaling-mu}
\mu(cg,c\tau)=\mu(g,\tau),
\qquad
\nu(cg)=\nu(g),
\end{equation}
and
\begin{equation}\label{eq:diffeo-mu}
\mu(\Phi^*g,\tau)=\mu(g,\tau),
\qquad
\nu(\Phi^*g)=\nu(g).
\end{equation}
\item[(2)] Let $g(t)$ solve the unnormalized Ricci flow
$\partial_tg=-2\Ric_g$ on a connected, closed, smooth Riemannian manifold. If $0<t<T$, then
\begin{equation}\label{eq:mu-monotonicity}
\mu(g(0),T)\leq\mu(g(t),T-t).
\end{equation}
Consequently, $\nu(g(t))$ is nondecreasing.
\end{itemize}

\end{proposition}
\begin{proof}
The proof could be found in \cite{Perelman} and related books about the Ricci flow, and we just sketch the proof of \eqref{eq:mu-monotonicity} here. Perelman \cite[Section~3, equations (3.1)-(3.4)]{Perelman} proved that, 
{for $$v(t)=\sqrt{u(t)} \text{ with }\int_Mv^2\dbarV {g(t)}=1,$$ and $\tau(t)>0$ with $\tau'(t)=-1$, one has
\begin{equation}\label{eq:Perelman-W-derivative}
\frac{d}{dt}\mathcal W_{g(t)}(v(t), \tau(t))=2\tau(t)\int_M\left|\Ric_{g(t)}-2\nabla^2\log v(t)-\frac{g(t)}{2\tau(t)}\right|^2v(t)^2\dbarV{g(t)}\geq0. 
\end{equation}
Here $u(t)$ solves
\begin{equation}
    \partial_tu=-\Delta_{g(t)}u+(R_{g(t)}-\overline{R}_{g(t)})u.
\end{equation}}Starting at time $t$ from a smooth minimizer for
$\mu(g(t),T-t)$ and solving the conjugate heat equation backward gives \eqref{eq:mu-monotonicity}. For $0\leq t_1<t_2$, by the time-shifted flow and \eqref{eq:mu-monotonicity}, we have, for every $\sigma>0$,
\[
\nu(g(t_1))\leq\mu(g(t_1),\sigma+t_2-t_1)
\leq\mu(g(t_2),\sigma).
\]
Taking the infimum over $\sigma$ proves monotonicity of $\nu$. 
\end{proof}

\begin{corollary}
\label{cor:normalized-monotonicity}
Let $\nu$ be defined as in \eqref{eq:mu-nu-definitions-2}. Let $g(s)$ be a smooth solution of the volume-normalized Ricci flow
\begin{equation}\label{eq:normalized-RF}
\partial_sg
=
-2\Ric_g+\frac{2}{n}\overline{R}_g\,g,
\end{equation}
Then $s\mapsto\nu(g(s))$ is nondecreasing.
\end{corollary}

\begin{proof}
Define
\begin{equation}\label{eq:alpha-def}
\alpha(s)
:=
\exp\left(-\frac{2}{n}\int_0^s\overline{R}_{g(r)}\,dr\right),
\qquad
 t(s):=\int_0^s\alpha(r)\,dr,
\end{equation}
and set
$$
\widetilde g(t(s)):=\alpha(s)g(s).
$$
Since the Ricci tensor as a $(0,2)$-tensor is unchanged by constant rescaling and $t'(s)=\alpha(s)$, we know
$$
\partial_t\widetilde g=-2\Ric_{\widetilde g}.
$$
Thus $\widetilde g$ is an unnormalized Ricci flow. For $0\le t_1<t_2$, by \eqref{eq:mu-monotonicity}, we have, for every $\tau>0$,
$$
\nu(\widetilde g(t_1))
\le
\mu(\widetilde g(t_1),\tau+t_2-t_1)
\le
\mu(\widetilde g(t_2),\tau).
$$
Taking the infimum over $\tau>0$, we obtain that $t\mapsto\nu(\widetilde g(t))$ is nondecreasing. By scale invariance \eqref{eq:scaling-mu},
$$
\nu(\widetilde g(t(s)))=\nu(\alpha(s)g(s))=\nu(g(s)).
$$
Because $t(s)$ is strictly increasing, the conclusion follows.
\end{proof}

\subsection{\bf P\'olya-Szeg\H{o} theorem}

Let $v\ge0$ be a measurable function defined on $M$. Its normalized distribution function is
$$
F_v(t):={\overline V}_g(\{v>t\}), \qquad t\ge0.
$$
where
${\overline V}_g(\{v>t\})$ denotes the measure of the set $\{x\in M|v(x)>t\}$ with respect to  $\dbarV g$.
The spherical decreasing rearrangement $v^*$ is a  radially decreasing function on $\mathbb S^n$, defined up to null sets, such that
\begin{equation}\label{eq:equimeasurable}
{\overline V}_{\Sph^n}(\{v^*>t\})=F_v(t),\qquad\text{for every }t\ge0.
\end{equation}
\begin{theorem}\label{thm:PolyaSzego}
Suppose that  $\Ric_g\geq(n-1)g$. If
$u\in H^1(M)$ is nonnegative, then $u^\star\in H^1(\Sph^n)$ and
\begin{equation}\label{eq:PS-energy}
\int_M|\nabla u|_g^2\,\dbarV g
\geq\int_{\Sph^n}|\nabla u^\star|_{g_{\Sph^n}}^2\,\dbarV {\Sph^n}.
\end{equation}
Moreover, for every Borel measurable function $F$ for which either side is integrable,
\begin{equation}\label{eq:PS-pushforward}
\int_MF(u)\,\dbarV g
=\int_{\Sph^n}F(u^\star)\,\dbarV {\Sph^n}.
\end{equation}
\end{theorem}

Theorem \ref{thm:PolyaSzego} is a direct consequence of Mondino-Semola \cite{MondinoSemola}(See also \cite{BerardMeyer}). To be more specific, one could apply \cite[Definition~1.3, Theorem~1.4, equation (1.5), and Proposition~3.5]{MondinoSemola} with
$$K=n-1,\qquad N=n,\qquad p=2,\qquad\Omega=M,$$
and note that a smooth Riemannian manifold satisfying $\Ric_g\geq(n-1)g$, equipped with
normalized volume, is an essentially nonbranching $\mathrm{CD}(n-1,n)$ space.

\subsection{\bf Beckner's inequality}
The following is actually Beckner's sharp spherical log-Sobolev inequality, the proof could be found in \cite[Theorem~2]{Beckner} or \cite[Corollary 1.1]{DEKM}.
\begin{theorem}\label{thm:Beckner}
For every $v\in H^1(\Sph^n)$ with
$\int_{\Sph^n}v^2\,\dbarV {\Sph^n}=1$,
\begin{equation}\label{eq:Beckner}
\int_{\Sph^n}v^2\log v^2\,\dbarV {\Sph^n}
\leq\frac2n\int_{\Sph^n}|\nabla v|_{g_{\Sph^n}}^2\,\dbarV {\Sph^n}.
\end{equation}
The constant $2/n$ is sharp, and constants attain equality.
\end{theorem}

\subsection{\bf Ma-Wang's theorem}
The following theorem is due to Ma-Wang \cite[Theorem~1.1]{MaWang}.

\begin{theorem}\label{thm:MaWang} 
Let $(M^n,g)$ be a connected, closed, smooth Riemannain manifold.
Set
\begin{equation}\label{eq:rho-def}
\rho_g(x):=\max\{0,(n-1)-\lambda_{\min}^g(\Ric_g)(x)\},
\end{equation}
where $\lambda_{\min}^g(\Ric_g)$ denotes  the smallest eigenvalue of the Ricci
tensor.  There exists a constant $\delta_{\MW}(n,p)>0$ such that  if 
\begin{equation}\label{eq:MW-hypothesis}
V_g(M)=|\Sph^n|,
\qquad
\int_M\rho_g^p\,dV_g<\delta_{\MW}(n,p)
\end{equation}
for some positive constant $p>\frac{n}{2}$, 
then the volume-normalized Ricci flow starting at $g$ exists for all forward time
and converges exponentially  to a round metric $g_\infty$
of constant sectional curvature one.
\end{theorem}

\section{Key Lemmas}\label{keylemma}\label{sec:keylemma}
As an application of Theorem \ref{thm:PolyaSzego}, we have the following entropy comparison result.
\begin{proposition}\label{prop:entropy-comparison}
Let $(M^n, g)$ be a connected, closed, smooth Riemannian manifold. Assume that \eqref{Kwong's-condition} holds.  Then, for every
$\tau>0$,
\begin{equation}\label{eq:mu-comparison}
\mu(g,\tau)\geq\mu(g_{\Sph^n},\tau)+\log\frac{V_g(M)}{|\Sph^n|}+n(n-1)\eps\tau.
\end{equation}
Consequently,
\begin{equation}\label{eq:nu-comparison-pre}
\nu(g)\geq\log\frac {V_g(M)}{|\Sph^n|}
+\inf_{\tau>0}\{\mu(g_{\Sph^n},\tau)+n(n-1)\eps\tau\}.
\end{equation}
\end{proposition}

\begin{proof}
Let $v\in H^1(M)$ be nonnegative and normalized as in
\eqref{eq:L2-normalization}. Let $v^\star$ be its spherical decreasing rearrangement. By Theorem~\ref{thm:PolyaSzego}, the Dirichlet energy $v^\star$ is no larger, and
\begin{equation}\label{eq:entropy-equimeasurable-v}
\int_Mv^2\,\dbarV g
=\int_{\Sph^n}(v^\star)^2\,\dbarV {\Sph^n},
\end{equation}
and
\begin{equation}\label{eq:entropy-equimeasurable}
\int_Mv^2\log v^2\,\dbarV g
=\int_{\Sph^n}(v^\star)^2\log(v^\star)^2\,\dbarV {\Sph^n}.
\end{equation}
The scalar curvature lower bound \eqref{Kwong's-condition} gives
\begin{equation}\label{eq:scalar-lower-test}
\int_MR_gv^2\,\dbarV g\geq n(n-1)(1+\eps).
\end{equation}
Substitution in \eqref{eq:W-probability} yields
$$
\W_g(v,\tau)\geq \W_{g_{\Sph^n}}(v^\star,\tau)
+\log\frac{V_g(M)}{|\Sph^n|}+n(n-1)\eps\tau.
$$
By \eqref{eq:H1-mu}, $v^\star$ is an admissible test function, so
$\W_{g_{\Sph^n}}(v^\star,\tau)\geq\mu(g_{\Sph^n},\tau)$. Taking the
infimum over $v$, we have
\eqref{eq:mu-comparison}. Finally taking the infimum over $\tau$ in \eqref{eq:mu-comparison},  we obtain \eqref{eq:nu-comparison-pre}.
\end{proof}

We define
\begin{equation}\label{eq:Cn}
C_n(\tau):=n(n-1)\tau+\log |\Sph^n|-\frac n2\log(4\pi\tau)-n.
\end{equation}
Actually, $C_n(\tau)=\W_{g_{\Sph^n}}(1,\tau)$. The following comparison statements for the round sphere are useful. 
\begin{lemma}\label{lem:sphere-mu}
The following statements hold.
\begin{itemize}
\item[(1)] For every $\tau\geq\frac{1}{2n}$,
\begin{equation}\label{eq:sphere-large}
\mu(g_{\Sph^n},\tau)=C_n(\tau).
\end{equation}
\item[(2)]For every $0<\tau\leq\frac{1}{2n}$,
\begin{equation}\label{eq:sphere-small}
\mu(g_{\Sph^n},\tau)\geq C_n\left(\frac{1}{2n}\right).
\end{equation}
\item[(3)]\begin{equation}\label{eq:sphere-nu}
\nu(g_{\Sph^n})=C_n\left(\frac{1}{2(n-1)}\right).
\end{equation}
\end{itemize}
\end{lemma}
\begin{proof}
For a nonnegative function $v\in H^1(\Sph^n)$ with $\int_{\Sph^n}v^2\dbarV {\Sph^n}=1$, we have
\begin{equation}\label{eq:W-minus-C}
\W_{g_{\Sph^n}}(v,\tau)-C_n(\tau)
=4\tau\int_{\Sph^n}|\nabla v|^2\,\dbarV {\Sph^n}
-\int_{\Sph^n}v^2\log v^2\,\dbarV {\Sph^n}.
\end{equation}
If $\tau\geq\frac{1}{2n}$, then by Theorem \ref{thm:Beckner} we know the right side of \eqref{eq:W-minus-C} is
nonnegative, and equality holds for $v\equiv1$. This proves
\eqref{eq:sphere-large}.

For the smaller scales $0<\tau\leq\frac{1}{2n}$, the unnormalized Ricci flow from the unit sphere is
\begin{equation}\label{eq:round-flow}
g_{\Sph^n}(t)=(1-2(n-1)t)g_{\Sph^n},
\qquad0\leq t<\frac{1}{2(n-1)}.
\end{equation}
Fix $0<\tau_1<\tau_2\leq\frac{1}{2n}$ and set
\begin{equation}\label{eq:t-choice}
t:=\frac{\tau_2-\tau_1}{1-2(n-1)\tau_1}.
\end{equation}
Because $2(n-1)\tau_2\le(n-1)/n<1$, one has $0<t<\tau_2$. Direct calculation gives
\begin{equation}\label{eq:t-algebra3}
\frac{\tau_2-t}{1-2(n-1)t}=\tau_1.
\end{equation}
Therefore, by Proposition \ref{prop:invariance}, we have
\begin{align*}
\mu(g_{\Sph^n},\tau_2)
&\leq\mu(g_{\Sph^n}(t),\tau_2-t)\\
&=\mu\left(g_{\Sph^n},
\frac{\tau_2-t}{1-2(n-1)t}\right)
=\mu(g_{\Sph^n},\tau_1).
\end{align*}
Here, we note the first inequality comes from Proposition \ref{prop:invariance}(2), and the equality of the second line comes from Proposition \ref{prop:invariance}(1).
Thus $\mu(\Sph^n,\cdot)$ is nonincreasing on $(0,\frac{1}{2n}]$, and
\eqref{eq:sphere-small} follows from \eqref{eq:sphere-large} at $\frac{1}{2n}$.

Finally,
\begin{equation}\label{eq:C-derivatives}
C_n'(\tau)=n(n-1)-\frac n{2\tau},
\qquad
C_n''(\tau)=\frac n{2\tau^2}>0.
\end{equation}
Hence $C_n(\tau)$ has its unique minimum at $\frac{1}{2(n-1)}$. Since
$\frac{1}{2(n-1)}>\frac{1}{2n}$, the large-scale formula \eqref{eq:sphere-large} gives
\begin{equation}\label{eq:Cntau}
    \mu\left(g_{\Sph^n},\frac{1}{2(n-1)}\right)=C_n\left(\frac{1}{2(n-1)}\right).
\end{equation} 
Notice that 
$$C_n\left(\frac{1}{2n}\right)-C_n\left(\frac{1}{2(n-1)}\right)=\frac n2\log\frac n{n-1}-\frac12>0.$$
Then, we obtain \eqref{eq:sphere-nu}.

Thus, we finish the proof.
\end{proof}
For brevity, set 
\begin{equation}\label{eq:gamma}
\gamma_n:=C_n\left(\frac{1}{2n}\right)-C_n\left(\frac{1}{2(n-1)}\right).
\end{equation}
A direct computation yields that 
\begin{equation}\label{eq:gamma+}\gamma_n=\frac n2\log\frac n{n-1}-\frac12>0.\end{equation}
Define \begin{equation}\label{eq:eps-ent}
\eps_{\ent}(n)
=\exp\left(\frac{2\gamma_n}{n}\right)-1
=\frac n{n-1}e^{-1/n}-1.
\end{equation}
One can easily check  
 \begin{equation}\label{eq:eps-ent-upper}
0<\eps_{\ent}(n)<\frac{1}{n-1}.\end{equation}
Define 
\begin{equation}\label{eq:Phi}
\Phi_n(\eps):=\inf_{\tau>0}
\{\mu(g_{\Sph^n},\tau)+n(n-1)\eps\tau\}.
\end{equation}
We obtain the precise value of such infimum.
\begin{proposition}
\label{prop:tilted-minimum}
If $0<\eps\leq\eps_{\ent}(n)$, then
\begin{equation}\label{eq:Phi-value}
\Phi_n(\eps)=\nu(g_{\Sph^n})+\frac n2\log(1+\eps).
\end{equation}
The unique minimizing scale is
\begin{equation}\label{eq:tau-eps}
\tau_\eps:=\frac1{2(n-1)(1+\eps)}.
\end{equation}
\end{proposition}
\begin{proof}
By \eqref{eq:eps-ent-upper}, $\tau_\eps>\frac{1}{2n}$. For $\tau\geq\frac{1}{2n}$,
Lemma~\ref{lem:sphere-mu} gives
\begin{equation}\label{eq:large-tilted}
\mu(g_{\Sph^n},\tau)+n(n-1)\eps\tau
=n(n-1)(1+\eps)\tau+\log |\Sph^n|
-\frac n2\log(4\pi\tau)-n.
\end{equation}
The second derivative of the right hand side of \eqref{eq:large-tilted} with respect to $\tau$ is $n/(2\tau^2)>0$, and its unique critical point is
$\tau_\eps$. Taking $\tau=\tau_\eps$ in \eqref{eq:large-tilted}, by \eqref{eq:sphere-nu} and \eqref{eq:Cn},  we have
\begin{equation}\label{eq:large-tilted-value}
\mu(g_{\Sph^n},\tau_\eps)+n(n-1)\eps\tau_\eps
=\nu(g_{\Sph^n})+\frac n2\log(1+\eps).
\end{equation}

For $0<\tau\leq\frac{1}{2n}$, Lemma~\ref{lem:sphere-mu} gives the strict estimate
\begin{equation}\label{eq:small-tilted-strict}
\mu(g_{\Sph^n},\tau)+n(n-1)\eps\tau
\geq C_n\left(\frac{1}{2n}\right)+n(n-1)\eps\tau>C_n\left(\frac{1}{2n}\right).
\end{equation}
On the other hand, by \eqref{eq:gamma} and \eqref{eq:gamma+} we have
\begin{align*}
\nu(g_{\Sph^n})+\frac n2\log(1+\eps)
&\leq\nu(g_{\Sph^n})+\frac n2\log(1+\eps_{\ent}(n))\\
&=\nu(g_{\Sph^n})+\gamma_n=C_n\left(\frac{1}{2n}\right).
\end{align*}
Thus every point of the small-scale branch is strictly larger than the value at
$\tau_\eps$, including when $\eps=\eps_{\ent}(n)$. This proves both
\eqref{eq:Phi-value} and uniqueness.
\end{proof}

\begin{corollary}
\label{cor:entropy-lower}
If $0<\eps\leq\eps_{\ent}(n)$ and \eqref{Kwong's-condition} holds, then
\begin{equation}\label{eq:entropy-sharp-factor}
\nu(g)\geq\nu(g_{\Sph^n})+\log\frac{V_g(M)}{|\Sph^n|}
+\frac n2\log(1+\eps).
\end{equation}
In particular, if
\begin{equation}\label{eq:volume-violation-entropy-0}
V_g(M)>|\Sph^n|(1+\eps)^{-n/2},
\end{equation}
then,
\begin{equation}\label{eq:volume-violation-entropy}
\nu(g)>\nu(g_{\Sph^n}).
\end{equation}
\end{corollary}

\begin{proof}
Inequality \eqref{eq:entropy-sharp-factor} is a direct consequence of Proposition \ref{prop:entropy-comparison} and
Proposition \ref{prop:tilted-minimum}. The implication
\eqref{eq:volume-violation-entropy} now follows immediately from \eqref{eq:entropy-sharp-factor} and \eqref{eq:volume-violation-entropy-0}.
\end{proof}

\begin{lemma}\label{lem:round-flow-upper}
Suppose $V_g(M)=|\Sph^n|$ and the volume-normalized Ricci flow $g(t)$ starting at
$g$ exists for all $t\geq0$ and converges smoothly  to a metric of constant sectional curvature one. Then
\begin{equation}\label{eq:nu-upper}
\nu(g)\leq\nu(g_{\Sph^n}).
\end{equation}
\end{lemma}

\begin{proof}
Corollary \ref{cor:normalized-monotonicity} gives
$\nu(g)\leq\nu(g(t))$.  The normalized flow preserves volume, so all metrics
have volume $|\Sph^n|$. For  $\tau_0=1/[2(n-1)]$,
using  \eqref{eq:W-probability},  we have
\begin{equation}\label{eq:constant-test-flow}
W_{g(t)}(1,\tau_0)
=\tau_0\overline R_{g(t)}+\log |\Sph^n|
-\frac n2\log(4\pi\tau_0)-n.
\end{equation}
Smooth convergence  implies 
$\overline R_{g(t)}\to n(n-1)$. Hence, using \eqref{eq:Cn} and Lemma \ref{lem:sphere-mu}, one has 
\begin{equation}\label{eq:constant-test-limit}
W_{g(t)}(1,\tau_0)\longrightarrow C_n(\tau_0)=\nu(g_{\Sph^n}).
\end{equation}
For every $t$, one has 
$$
\nu(g)\leq\nu(g(t))\leq\mu(g(t),\tau_0)
\leq W_{g(t)}(1,\tau_0).
$$
Letting $t\to\infty$ proves \eqref{eq:nu-upper}.  This completes the proof.
\end{proof}
{Now we consider the equality condition in above entropy monotonicity. For the reader's convenience, we put the proof below.
\begin{lemma}\label{lemma:equalrigid}
    Under the assumptions of Lemma \ref{lem:round-flow-upper}, suppose in addition that $\nu(g(0))=\nu(g_{\Sph^n})$, then $$(M, g(0))\cong(\Sph^n, g_{\Sph^n}).$$
\end{lemma}
\begin{proof}
    Fix $s_0>0$. The time-shifted flow $s\to g(s_0+s)$ satisfies the same hypotheses as the original flow. Applying Corollary \ref{cor:normalized-monotonicity} and Lemma \ref{lem:round-flow-upper}, we have 
    \begin{equation}
        \nu(g_{\Sph^n})=\nu(g(0))\leq \nu(g(s_0))\leq\nu(g_{\Sph^n}).
    \end{equation}
Therefore, \begin{equation}\label{eq:nunusn}
\nu(g(s))\equiv\nu(g_{\Sph^n})\text{ for all } s\geq0.
\end{equation}
Fix $S>0$ and define $\alpha(s)$ and $t(s)$ as in \eqref{eq:alpha-def} for $0\le s\le S$. Let $\widetilde{g}(t(s))=\alpha(s)g(s)$, then, $\widetilde{g}(t)$ is an unnormalized Ricci flow,
$$\partial_t \widetilde{g}=-2\Ric_{ \widetilde{g}(t)},\qquad \widetilde{g}(0)=g(0).$$
By scale invariance \eqref{eq:scaling-mu} and \eqref{eq:nunusn}, we have 
 \begin{equation}\label{eq:widnusn}
\nu(\widetilde{g}(t))\equiv\nu(g_{\Sph^n})\text{ for all } 0\leq t\leq T:=t(S).
\end{equation}
\textbf{Claim:} At time $T$, there are $0<\tau_T<+\infty$ and $v_T> 0$ such that 
\begin{equation}
\nu(\widetilde{g}(T))=\W_{\widetilde{g}(T)}(v_T, \tau_T),
\end{equation}
where $$\int_Mv^2_T\dbarV{\widetilde{g}(T)}=1.$$
By \eqref{eq:Cn}, \eqref{eq:sphere-nu} and the corresponding inequality of Stirling's approximation, we know 
\begin{equation}\label{eq:neganu}
\nu(\widetilde{g}(T))=\nu(g_{\Sph^n})=\log\left(|\Sph^n|\left(\frac{n-1}{2\pi e}\right)^{\frac{n}{2}}\right)<0.
\end{equation}
Let $\lambda_{1}$ be the first eigenvalue of $-4\Delta_{\widetilde{g}(T)}+R_{\widetilde{g}(T)}$, then $\lambda_1>0$. Otherwise, assume $\lambda_1\le0$ and let $\phi>0$ be the corresponding eigenfunction with $$\int_M\phi^2\dbarV{\widetilde g(T)}=1.$$
Then, $$\W_{\widetilde g(T)}(\phi, \tau)=\lambda_1\tau-\frac{n}{2}\log(4\pi\tau)-\int_M\phi^2\log(\phi^2)\dbarV{\widetilde{g}(T)}+\log V_{\widetilde{g}(T)}-n.$$
Since $\lambda_1\le0$, we have $\W_{\tilde g(T)}(\phi, \tau)\to-\infty$ as $\tau\to +\infty$, which implies $\nu(\widetilde{g}(T))=-\infty$. This contradicts \eqref{eq:neganu}.

By \cite[Section 3.1]{Perelman} and note $M$ is closed and $\lambda_1>0$, we know
\begin{equation}\label{eq:mu0}
\lim_{\tau\to0}\mu(\widetilde{g}(T), \tau)=0,
\end{equation}
and 
\begin{equation}\label{eq:muinfty}
    \lim_{\tau\to+\infty}\mu(\widetilde{g}(T), \tau)=+\infty.
\end{equation}
By the continuity of $\mu(g, \tau)$ with respect to $\tau$ on the fixed closed manifold and \eqref{eq:neganu}, we know the infimum of $\mu(\widetilde{g}(T), \tau)$ is strictly negative, and there is $\tau_T\in(0, +\infty)$ such that 
$$\mu(\widetilde{g}(T), \tau_T)=\nu(\widetilde{g}(T)).$$
For this fixed scale $\tau_T$, the direct method and regularity theory of elliptic equations give a smooth minimizer $v_T\ge 0$ satisfying 
\begin{equation}\label{W(g(T))=nug(T)}
    \int_M v_T^2\dbarV{\widetilde{g}(T)}=1,\qquad \W_{\widetilde{g}(T)}(v_T,\tau_T)=\nu(\widetilde{g}(T)).
\end{equation}
The Euler–Lagrange equation and the strong maximum principle imply $v_T>0$. Thus, we complete the proof of the claim.

Next, set $u_T=v_T^2$. Starting from $u_T$ at time $T$, solve backward the
normalized conjugate heat equation
\begin{equation}\label{eq:q-backward}
 \partial_tu=-\Delta_{\widetilde g(t)}u
 +\bigl(R_{\widetilde g(t)}
 -\overline R_{\widetilde g(t)}\bigr)u,
\end{equation}
and put
\begin{equation}\label{eq:v-tau-backward}
 v(t):=\sqrt{u(t)},
 \qquad
 \tau(t):=\tau_T+T-t.
\end{equation}
The solution remains positive and normalized. Integrating
\eqref{eq:Perelman-W-derivative} from $0$ to $T$ gives
\begin{align}
 &\W_{\widetilde g(T)}(v_T,\tau_T)
 -\W_{\widetilde g(0)}(v(0),\tau(0)) \notag\\
 &\quad=2\int_0^T\tau(t)\int_M
 \left|\Ric_{\widetilde g(t)}-2\nabla^2\log v(t)
 -\frac{\widetilde g(t)}{2\tau(t)}\right|^2
 v(t)^2\,\dbarV {\widetilde g(t)}\,dt.       \label{eq:integrated-square}
\end{align}
Notice that 
\[
 \W_{\widetilde g(0)}(v(0),\tau(0))
 \geq\nu(\widetilde g(0)).
\]
Using \eqref{eq:nunusn} and the fact $v(t)>0$,   we have
\begin{equation}\label{eq:v-shrinker-all-times}
 \Ric_{\widetilde g(t)}-2\nabla^2\log v(t)
 =\frac{\widetilde g(t)}{2\tau(t)}
 \qquad\text{on }M\times[0,T].
\end{equation}
In particular, recalling $\widetilde{g}(0)=g(0)$, one has
\begin{equation}
     \Ric_{g(0)}-2\nabla^2\log v(0)
 =\frac{ g(0)}{2\tau(0)}
\end{equation}
Thus, there exists a family of  diffeomorphism $\Phi_t$ such that 
\begin{equation}\label{pullback}
    {g}(t)=\Phi_t^*(g(0)).
\end{equation}
Since $g(t)\to g_\infty$ smoothly and $g_\infty$ has constant sectional curvature one,  $g(0)$ has constant sectional curvature one. Noticing  that $V_{g(0)}(M)=|\Sph^n|$, the Bishop-Gromov theorem yields that $(M,g(0))$ is isometric to standard sphere.

\end{proof}
}

\section{Proof of the main result}\label{sec:proof}

Up to scaling, Theorem \ref{thm:main-theorem} is equivalent to the following theorem.

\begin{theorem}
        Let $(M^n,g)$ be a connected, complete,  smooth Riemannian manifold with $n\geq 3$. There exists a positive constant $\tilde\varepsilon_n>0$ such that, if $0<\varepsilon\leq \tilde\varepsilon_n$ and \eqref{Kwong's-condition} holds,   then \eqref{Bray's-conjecture} holds. {Moreover, the equality in \eqref{Bray's-conjecture} holds if and only if $(M, g)\cong(\Sph^n, (1+\eps)^{-1}g_{\Sph^n})$.}
\end{theorem}
\begin{proof}
First, Bonnet-Myers thereom implies that $M$ is compact. 
Let $\delta_{MW}(n,n)$ be the constant from Theorem \ref{thm:MaWang} and $\varepsilon_{\ent}(n)$ in \eqref{eq:eps-ent}, and define
\begin{equation}\label{eq:epsilon-n-definition}
\tilde\varepsilon_n
:=\min\left\{
\varepsilon_{\ent}(n),
\left(
\frac{\delta_{\MW}(n,n)}{|\Sph^n|(n-1)^n}
\right)^{1/n}
\right\}.
\end{equation}
Both terms inside the minimum are positive, so $\tilde\eps_n>0$.

Fix $0<\eps\leq\tilde \eps_n$ and suppose, for contradiction, that
\begin{equation}\label{eq:contradiction-volume}
V_g(M)>|\Sph^n|(1+\eps)^{-n/2}.
\end{equation}
Since $0<\eps\leq\eps_{\ent}(n)$, Corollary~\ref{cor:entropy-lower} gives
\begin{equation}\label{eq:main-lower}
\nu(g)>\nu(g_{\Sph^n}).
\end{equation}
Bishop's theorem gives $V_g(M)\leq |\Sph^n|$. Define
\begin{equation}\label{eq:main-rescale}
c:=\left(\frac{|\Sph^n|}{V_g(M)}\right)^{2/n},
\qquad h:=cg.
\end{equation}
Thus, one has 
\begin{equation}\label{eq:c-range}
V_h(M)=|\Sph^n|,
\qquad 1\leq c<1+\eps.
\end{equation}
Under constant rescaling, the Ricci tensor viewed as a $(0,2)$ tensor, is unchanged, so
\begin{equation}\label{eq:rescaled-Ricci}
\Ric_h=\Ric_g\geq(n-1)g=\frac{n-1}{c}h.
\end{equation}
The deficit \eqref{eq:rho-def} therefore satisfies
\begin{equation}\label{eq:defect-pointwise}
0\leq\rho_h\leq(n-1)\left(1-\frac1c\right)
<(n-1)\frac{\eps}{1+\eps}.
\end{equation}
Consequently,
\begin{equation}\label{eq:defect-integral}
\int_M\rho_h^n\,dV_h
<|\Sph^n|(n-1)^n\left(\frac{\eps}{1+\eps}\right)^n<\delta_{\MW}(n,n).
\end{equation}
Theorem \ref{thm:MaWang} applies to $h$, and
Lemma \ref{lem:round-flow-upper} gives $\nu(h)\leq\nu(g_{\Sph^n})$. Scale
invariance \eqref{eq:scaling-mu} gives
\begin{equation}\label{eq:main-upper}
\nu(g)=\nu(h)\leq\nu(g_{\Sph^n}),
\end{equation}
contradicting \eqref{eq:main-lower}. Thus \eqref{Bray's-conjecture} holds.

{Now we assume the equality in \eqref{Bray's-conjecture}: 
\begin{equation*}
V_g(M)=(1+\eps)^{-\frac{n}{2}}|\Sph^n|.    
\end{equation*}
The entropy comparison Corollary \ref{cor:entropy-lower} yields
\begin{equation}\label{eq:equalnu}
\nu(g)\geq\nu(g_{\Sph^n})+\log(1+\eps)^{-\frac{n}{2}}
+\frac n2\log(1+\eps)=\nu(g_{\Sph^n}).
\end{equation}
Now by definition in \eqref{eq:main-rescale}, we know $c=1+\eps$ and $h=(1+\eps)g$. In this case,
$$0\leq \rho_h\leq (n-1)\frac{\eps}{1+\eps}.$$
Note the last inequality in \eqref{eq:defect-integral} is strict, we can still apply Theorem \ref{thm:MaWang} to $h$ and finally get \eqref{eq:main-upper}.
Combining \eqref{eq:main-upper} and \eqref{eq:equalnu}, we have $\nu(h)=\nu(g_{\Sph^n}).$ By Lemma \ref{lemma:equalrigid}, we obtain $(M, h)\cong(\Sph^n, g_{\Sph^n})$, and therefor, $(M, g)\cong(\Sph^n, (1+\eps)^{-1}g_{\Sph^n})$. Thus, we complete the proof.}

\end{proof}

\section{Applications to  volume comparison about $Q$-curvature}\label{sec:application}

For a Riemannian manifold $(M^n,g)$ with dimension $n\geq 3$, the fourth-order Branson $Q$-curvature is given by
\begin{equation}\label{Q-def}
    Q_g=-\frac{1}{2(n-1)}\Delta_gR_g-\frac{2}{(n-2)^2}\left|\operatorname{Ric}_g-\frac{R_g}{n}g\right|_g^2+\frac{n^2-4}{8n(n-1)^2}R_g^2.
\end{equation}
Using such definition, if $\operatorname{Ric}_g\equiv (n-1)g$, it is easy to show that
$$
Q_g\equiv \frac{n(n^2-4)}{8}.
$$
Lin and Yuan \cite{LinYuan1,LinYuan2} studied the deformation of the $Q$-curvature and obtained volume comparison theorems using it. More background material can be found in these papers and the references therein.

More precisely, in Theorem 1.7 and Corollary 1.10 of \cite{LinYuan2}, for a stable Einstein manifold $(M^n,\bar g)$ with $\operatorname{Ric}_{\bar g}=(n-1)\bar g$, Lin and Yuan showed that if the metric $g$ is sufficiently close to $\bar g$ in the $C^4$-topology and
$$
Q_g\geq \frac{n(n^2-4)}{8},
$$
then $V_g(M^n)\leq V_{\bar g}(M^n)$. As an interesting application of Bray's conjecture in the context of $Q$-curvature, we obtain the following theorem.

\begin{theorem}
Let $(M^n,g)$ be a connected, complete, smooth Riemannian manifold of dimension $n\geq 3$. There exists a positive constant $\varepsilon_n<1$ such that if the Ricci curvature satisfies
$$
\operatorname{Ric}_g\geq \varepsilon_n (n-1)g
$$
and the $Q$-curvature defined in \eqref{Q-def} satisfies
$$
Q_g\geq \frac{n(n^2-4)}{8},
$$
then
$$
V_g(M^n)\leq |\mathbb{S}^n|,
$$
and the equality holds if and only if $(M^n, g)$ is isometric to standard sphere.
\end{theorem}

\begin{proof}
First, Bonnet-Myers theorem implies that the manifold must be compact.
From the Ricci curvature assumption, taking the trace gives $R_g>0$. Let $x_0$ be a minimum point of $R_g$ on $M^n$. Then $\Delta_g R_g(x_0)\geq 0$. Using \eqref{Q-def} and the assumption on $Q$-curvature, we get
$$
\frac{n(n^2-4)}{8}\leq Q_g(x_0)\leq \frac{n^2-4}{8n(n-1)^2}R_g(x_0)^2,
$$
which yields
$$
R_g(x_0)\geq n(n-1).
$$
Thus $R_g\geq n(n-1)$ everywhere. Applying Theorem \ref{thm:main-theorem} finishes the proof.
\end{proof}

\begin{remark}
For a complete four-dimensional manifold $(M^4,g)$, we only need to assume $R_g\geq 0$ and $Q_g\geq 6$.  By using Bonnet-Myers type theorem in \cite[Theorem 1.1]{jiangliwang}, the manifold must be compact. Then one can obtain $V_g(M^4)\leq |\mathbb{S}^4|$ by using Gursky's inequality \cite{Gursky} (see Theorem 1.6 of \cite{liwei}).
\end{remark}

\bibliography{bib}
\bibliographystyle{plain}

\end{document}